\documentclass[12pt]{amsart}
\usepackage{amsmath,amsthm,amsfonts,amssymb,amscd,euscript}
\usepackage[T1]{fontenc}
\usepackage{color}
\usepackage{graphics}
\usepackage{tikz}
\input{xy}
\xyoption{all}

\newlength{\defbaselineskip}
\theoremstyle{plain}
\newtheorem{theorem}{Theorem}[section]

\newtheorem{proposition}[theorem]{Proposition}
\newtheorem{corollary}[theorem]{Corollary}
\newtheorem{lemma}[theorem]{Lemma}

\theoremstyle{definition}
\newtheorem{example}[theorem]{Example}
\newtheorem{definition}[theorem]{Definition}

\newcommand{\supp}{\operatorname{supp}}

\newcommand{\R}{\mathbb{R}}

\newcommand{\cR}{\mathcal{R}}
\newcommand{\ff}{F}
\newcommand{\GG}{G}

\newtheorem{theoremalpha}{Theorem}

\newtheorem{conjecturealpha}[theoremalpha]{Conjecture}

\theoremstyle{plain}

\theoremstyle{definition}

\numberwithin{equation}{section}

\newtheorem*{rep@theorem}{\rep@title} \newcommand{\newreptheorem}[2]{%
\newenvironment{rep#1}[1]{%
\def\rep@title{\bf #2 \ref{##1}}%
\begin{rep@theorem} }%
{\end{rep@theorem} } }
\makeatother
\newreptheorem{theorem}{Theorem}
\newreptheorem{conjecture}{Conjecture}

\begin{document}

\title{On the two-dimensional Jacobian conjecture: Magnus' formula revisited, I}

\author{William E. Hurst}
\address{Department of Mathematics, University of Alabama, Tuscaloosa, AL 35487, U.S.A.}
\email{wehurst@crimson.ua.edu}

\author{Kyungyong Lee}
\address{Department of Mathematics, University of Alabama,
	Tuscaloosa, AL 35487, U.S.A. 
	and Korea Institute for Advanced Study, Seoul 02455, Republic of Korea}
\email{kyungyong.lee@ua.edu; klee1@kias.re.kr}

\author{Li Li}
\address{Department of Mathematics and Statistics,
Oakland University, 
Rochester, MI 48309, U.S.A.}
\email{li2345@oakland.edu}

\author{George D. Nasr}
\address{Department of Mathematics, University of Oregon, Eugene, OR 97403, U.S.A.}
\email{gdnasr@uoregon.edu}

	\dedicatory{To the memory of Shreeram Shankar Abhyankar}
	
\begin{abstract}
Let $K$ be an algebraically closed field of characteristic 0. When the Jacobian $({\partial f}/{\partial x})({\partial g}/{\partial y}) - ({\partial g}/{\partial x})({\partial f}/{\partial y})$ is a constant for $f,g\in K[x,y]$, Magnus' formula from \cite{Magnus1} describes the relations between the homogeneous degree pieces $f_i$'s and $g_i$'s. 
We show a more general version of Magnus' formula and  prove a special case of  the two-dimensional Jacobian conjecture as its application. 
\end{abstract}

\thanks{This paper grew out of an undergraduate research project for WH, which had been initiated by 
the Randall Research Scholars Program  at the University of Alabama. KL was supported by the University of Alabama, Korea Institute for Advanced Study, and the NSF grant DMS 2042786. GDN is supported by the NSF FRG grant, Grant Number DMS-2053243.} 

\maketitle

\section{introduction}\label{section_intro}
The Jacobian conjecture, raised by Keller \cite{Keller}, has been studied by many mathematicians: a partial list of related results includes \cite{AM,A,AdEs,ApOn,BCW,BK,dBY,NVC,CW1,CW2,CZ,Dru,EssenTutaj,EssenWZ,Gwo,Hub,JZ,Kire,LM,M,MU,MO,Nagata,NaBa,Wang,Yag,Yu}.  A survey is given in \cite{Essen,vdEssen}. In this paper we exclusively deal with the plane case. Hence whenever we write the Jacobian conjecture, we mean the two-dimensional Jacobian conjecture.

Let $K$ be an algebraically closed field of characteristic 0, and let $\mathcal{R}=K[x,y]$.

\noindent \textbf{Jacobian conjecture.} 
\emph{Let $f(x,y),g(x,y)\in \mathcal{R}$. Consider the  polynomial map }$\pi: \mathcal{R}\longrightarrow \mathcal{R}$\emph{ given by}
$\pi(x)=f(x,y)$ and $\pi(y)=g(x,y)$. \emph{If the Jacobian of the map }
$$\det \begin{pmatrix}{\partial f}/{\partial x} & {\partial g}/{\partial x}\\ {\partial f}/{\partial y} &{\partial g}/{\partial y}\end{pmatrix}$$
\emph{is a non-zero constant, then the map is bijective.}

For simplicity, let $[f, g] :=\det \begin{pmatrix}{\partial f}/{\partial x} & {\partial g}/{\partial x}\\ {\partial f}/{\partial y} &{\partial g}/{\partial y}\end{pmatrix}\in \mathcal{R}$ for any pair of polynomials $f,g\in \mathcal{R}$. Similarly $[f,g]$ is defined for $f,g\in \mathcal{R}[[t]]=K[[t]][x,y]$. 

A useful tool to study this conjecture is the Newton polygon. One source for this is \cite{CW2}, but we redefine it here. Let \[f=\displaystyle\sum_{i,j\geq 0} f_{ij}x^iy^j\] be a polynomial in $\mathcal{R}$. The \textit{support} of $f$ is defined as 
\[\supp(f)=\{(i,j) \mid f_{ij}\neq 0\}\subset \mathbb{Z}^2\subset \R^2.\]
The \textit{Newton polygon} for $f$, which we denote $N(f)$, is defined to be the convex hull of $\supp(f)$ \footnote{In \cite{CW2}, it was defined as the convex hull of $\supp(f)\cup \{(0,0)\}$ in $\R^2$, which is different from  our definition.} in $\R^2$. Note that $N(f)\subset \mathbb{R}_{\ge0}^2$. See Figure \ref{fig:newton_ex} for an example of a Newton polygon. The support and Newton polygon of a Laurent polynomial in $K[x^{\pm 1},y^{\pm 1}]$ are similarly defined.

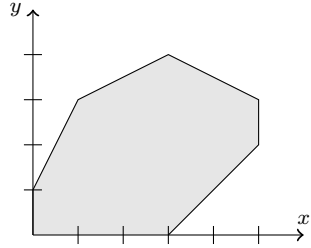
\begin{figure}[h]
\begin{center}
\begin{tikzpicture}[scale=0.60]
\usetikzlibrary{patterns}
\filldraw [black!10]  (0,0)--(0,1)--(1,3)--(3,4)--(5,3)--(5,2)--(3,0);
\draw [black]  (0,0)--(0,1)--(1,3)--(3,4)--(5,3)--(5,2)--(3,0);
\draw[->] (0,0) -- (6,0)
node[above] {\tiny $x$};
\draw[->] (0,0) -- (0,5)
node[left] {\tiny $y$};
\draw (-0.2,1)--(0.2,1);
\draw (-0.2,2)--(0.2,2);
\draw (-0.2,3)--(0.2,3);
\draw (-0.2,4)--(0.2,4);
\draw(1,-0.2)--(1,0.2);
\draw(2,-0.2)--(2,0.2);
\draw(3,-0.2)--(3,0.2);
\draw(4,-0.2)--(4,0.2);
\draw(5,-0.2)--(5,0.2);
\end{tikzpicture}
\end{center}
\caption{$N(f)$ for $f=y+7xy^3+\sqrt{7} x^3y^4-4x^5y^3+2x^5y^2-{1\over 2}x^3+xy+1$.}
\label{fig:newton_ex}
\end{figure}

 Throughout this paper whenever we consider pairs of polynomials $f,g$ with $[f,g]\in K$, we only consider such pairs  for which both  $N(f)$ and $N(g)$ contain $(1,0),(0,1)$, and $(0,0)$. As long as $\deg(f)$ and $\deg(g)$ are positive, it is always possible to obtain such a pair by adding a generic constant to $f$ and $g$ and applying some linear change of variables $x$ and $y$, which does not change $[f, g]$. 
 
  It is known that the following conjecture implies the Jacobian conjecture. (For instance, see \cite[Theorem 10.2.23]{vdEssen}.) 

\begin{conjecturealpha}\label{Jac_conj}
Let $a,b\in \mathbb{Z}_{>0}$ be relatively prime. 
Suppose that $\ff,\GG\in \mathcal{R}$ satisfy the following:

\noindent\emph{(1)} $[\ff,\GG]  \in  K$;

\noindent\emph{(2)} $\{(1,0),(0,1),(0,0)\}\subset N(\ff)\cap N(G)$ and $N(\ff)$ is similar to $N(\GG)$ with
the origin as center of similarity and with ratio $\deg(\ff) : \deg(\GG) = a : b$; and

\noindent\emph{(3)} $\min(a,b)\ge 2$.

\noindent Then  $[\ff, \GG]=0$. 
\end{conjecturealpha}

Let $W=\{(u,v)\in \mathbb{Z}^2 \, : \, u>0\text{ or }v>0,\text{ and }\gcd(|u|,|v|)=1\}$. An element  $w = (u, v) \in W$ is called a \emph{direction}. To each such a direction we consider its $w$-grading on $\mathcal{R}$ by defining the $w$-degree of the monomial $x^iy^j$ as $n=ui + vj$. Define $\mathcal{R}_n\subset \mathcal{R}$ to be the $K$-subspace generated by monomials of $w$-degree $n$.
Then $\mathcal{R} = \oplus_{n\in \mathbb{Z}} \mathcal{R}_n$. 
A non-zero element $P$ of $\mathcal{R}_n$ is called a $w$-homogeneous element of $\mathcal{R}$; 
the integer $n$ is called the $w$-degree of $P$ and is denoted $w$-$\deg(P)$. The element of the highest $w$-degree in the homogeneous decomposition of
a non-zero polynomial $P$ is called its $w$-leading form and is denoted by $P_+$. The
$w$-degree of $P$ is by definition $w$-$\deg(P_+)$.

In \cite[Theorem 1]{Magnus1}, Magnus produced a formula which inspired much of the work for this paper. 
His formula was published almost 70 years ago but has not been used in almost any paper but \cite{Magnus2}. Even in \cite{Magnus2}, only small piece of information from the formula was utilized.  The first main result in our paper is a more general version of Magnus' formula, as given in Theorem~\ref{Magnus_thm}.
In what follows, the binomial coefficient ${A\choose B}$  is defined by \[{A\choose B}:={A(A-1)\cdots(A-B+1)\over B!}\] for any real number $A$ and any nonnegative integer $B$.

\begin{theorem}\label{Magnus_thm}
Suppose $[F, G] \in  K$. For any direction $w=(u,v) \in W$, let $d=w\text{-}\deg(F_+)$ and $e=w\text{-}\deg(G_+)$. Write the $w$-homogeneous degree decompositions $F=\sum_{i\leq d} F_i$ and $G=\sum_{i\leq e} G_i$. Then 
there exists a unique \footnote{Note that there is some ambiguity of the notation $F_d^{1/r}$, since it is unique up to an $r$-th root of unity. We fix a choice of $F_d^{1/r}$. Then the fractional power $F_d^{c/r}:=(F_d^{1/r})^c$ is nonambiguous for any integer $c$.} sequence of constants $c_0,c_1,...,c_{d+e-u-v-1}\in K$ such that $c_0\neq 0$ and 
\begin{equation}\label{Magnus_formula}
G_{e-\mu} = \sum_{\gamma=0}^{\mu} c_\gamma \sum {(e-\gamma)/d \choose \sum_{\alpha\le d-1} \nu_{\gamma,\alpha} } \frac{\left(\sum_{\alpha\le d-1} \nu_{\gamma,\alpha}\right)!}{\prod_{\alpha\le d-1} \nu_{\gamma,\alpha}!} F_d^{(e-\gamma)/d - \sum_{\alpha\le d-1} \nu_{\gamma,\alpha}} \prod_{\alpha\le d-1} F_\alpha^{\nu_{\gamma,\alpha}}
\end{equation}
for every integer $\mu\in\{0,1,...,d+e-u-v-1\}$, where the inner sum is to run over all combinations of non-negative integers $\nu_{\gamma,\alpha}$ satisfying 
$\sum_{\alpha\le d-1} (d-\alpha)\nu_{\gamma,\alpha}=\mu-\gamma$.  Furthermore, $c_\gamma=0$ if $r(e-\gamma)/d\notin\mathbb{Z}$, where $r\in\mathbb{Z}_{>0}$ is the largest integer such that $F_d^{1/r}\in K[x,y]$. 
\end{theorem}

 In a series of forthcoming papers, we will make progress toward Conjecture~\ref{Jac_conj}, using Theorem~\ref{Magnus_thm}. The purpose of this paper is to write out a proof of Theorem~\ref{Magnus_thm}\footnote{In \cite{Magnus1}, a detailed proof was not given. Moreover, the original statement in \cite{Magnus1} was written only for $w=(1,1)$, and did not contain the statement that starts with "Furthermore", which will play a pivotal role in a series of our papers including this one.} and   illustrate how useful this theorem is.

Let $F$ and $G$ satisfy the assumptions (1) and (2) in Conjecture~\ref{Jac_conj}. Then $[\ff, \GG]\in K$ implies $\ff_+^{1/a}\in K[x,y]$ for any direction $w \in W$ (for instance, see \cite{A,ApOn,NaBa}). 
Let $\mathcal{T}$ be the set of  polynomials $f\in K[x,y]$ such that $N(f)$ contains exactly two distinct lattice points, i.e., $N(f)$ is a line segment containing no lattice points other than its endpoints.

\begin{definition}
We say that the pair $(\ff,\GG)$ has \emph{generic boundaries} if it satisfies (1) and (2) in Conjecture~\ref{Jac_conj}, and  the polynomial $\ff_+^{1/a}$ is not divisible by the  square of any polynomial in $\mathcal{T}$ for any  direction $w \in W$. \footnote{Note that the latter condition is equivalent to that $\ff_+^{1/a}$ does not have as a divisor any square of a non-monomial polynomial, since $K$ is algebraically closed.}
\end{definition}

The following is the second main result in this paper.

\begin{theorem}\label{main_thm}
If $(\ff,\GG)$ has generic boundaries with $a=2$, then Conjecture~\ref{Jac_conj} is true. More precisely, we have 
$$\ff={P}^2 + u_0$$
for some ${P}\in K[x,y]$ and some $u_0\in K$. In particular, $[\ff,\GG]=0$.
\end{theorem}
 
For a real number $r\in \mathbb{R}$ and a subset $S\subseteq\mathbb{R}^2$, denote $rS:=\{rs\ : s\in S\}\subseteq\mathbb{R}^2$.  
 
\begin{corollary}
Suppose that each edge of $\frac{1}{a}N(\ff)$ contains either the origin or no lattice points other than its endpoints. If $a=2$ then Conjecture~\ref{Jac_conj} is true. 
\end{corollary}

\section{Magnus' Formula Revisited}
The goal of this section is to prove Theorem \ref{Magnus_thm} and present Proposition~\ref{application_of_Magnus}, a useful application of this Theorem.
We start by reinterpreting  the equality \eqref{Magnus_formula}  in Theorem~\ref{Magnus_thm} as follows. Recall that $\cR=K[x,y]$, where $K$ is an algebraically closed field of characteristic 0. 
For any $\widetilde{F}\in \mathcal{R}[[t]]$, denote  
$$[\widetilde{F}]_{t^i}=\textrm{ the coefficient of $t^i$ in $\widetilde{F}$},$$
which is a polynomial in $x$ and $y$.

Recall the generalized multinomial theorem in the formal power series ring $K[[x_1,\dots,x_n]]$: for $A\in\mathbb{Q}$,
$$(1+x_1+\cdots+x_{n})^A=\sum_{v_1,\dots,v_n\in\mathbb{Z}_{\ge0}}\frac{A(A-1)\cdots(A-\sum_{i=1}^n v_i) }{\prod_{i=1}^{n} v_i!}x_1^{v_1}\cdots x_n^{v_n}.$$
Consider a variation of this. For $x_1,\dots,x_n\in \mathcal{R}$, we have the following expansion in the ring $\mathcal{R}[[t]]$:
$$(1+x_1t+\cdots+x_{n}t^n)^A=\sum_{v_1,\dots,v_n\in\mathbb{Z}_{\ge0}}\frac{A(A-1)\cdots(A-\sum_{i=1}^n v_i) }{\prod_{i=1}^{n} v_i!}x_1^{v_1}\cdots x_n^{v_n}t^{v_1+2v_2+\cdots+nv_n}.$$
In general, for $x_0,\dots,x_n\in \mathcal{R}$ and for $A=a/b$ where $a\in\mathbb{Z}$, $b\in\mathbb{Z}_{>0}$,   
we have the following identity in the ring $\mathcal{R}[x_0^{\pm 1/b}][[t]]$ (where we fix a choice of $x_0^{1/b}$): 
\begin{equation}\label{eq:xA}
(x_0+x_1t+\cdots+x_{n}t^n)^A=\sum_{v_1,\dots,v_n\in\mathbb{Z}_{\ge0}}\frac{A\cdots(A-\sum_{i=1}^n v_i) }
{\prod_{i=1}^{n} v_i!} x_0^{A-\sum_{i=1}^n v_i}x_1^{v_1}\cdots x_n^{v_n}t^{v_1+\cdots+nv_n}.
\end{equation}

\begin{lemma}\label{reinterpret Magnus formula}
The equality \eqref{Magnus_formula} can be rewritten as the following equality in $\mathcal{R}[F_d^{\pm 1/d}]$:
\begin{equation}\label{eq:gemu_variation}
G_{e-\mu} 
= \sum_{\gamma=0}^{\mu} c_\gamma 
\bigg[\bigg(F_d+F_{d-1}t+F_{d-2}t^2+\cdots\bigg)^{(e-\gamma)/d}\bigg]_{t^{\mu-\gamma}}.
\end{equation}
\end{lemma}
\begin{proof}
Let $A=(e-\gamma)/d$ and $s=\sum_{\alpha\le d-1} \nu_{\gamma,\alpha}$. Then
$${(e-\gamma)/d \choose \sum_{\alpha\le d-1} \nu_{\gamma,\alpha} } \frac{\left(\sum_{\alpha\le d-1} \nu_{\gamma,\alpha}\right)!}{\prod_{\alpha\le d-1} \nu_{\gamma,\alpha}!}
=
\frac{A(A-1)\cdots(A-s+1) }{ s! } \frac{s!}{\prod_{\alpha\le d-1} \nu_{\gamma,\alpha}!}
=
\frac{A\cdots(A-s+1) }{\prod_{\alpha\le d-1} \nu_{\gamma,\alpha}!}.
$$
So the right side of \eqref{Magnus_formula}, without the constraint $
\sum_{\alpha\le d-1} (d-\alpha)\nu_{\gamma,\alpha}=\mu-\gamma
$, is 
$$
\sum_{\gamma=0}^{\mu} c_\gamma 
\bigg(F_d+F_{d-1}t+F_{d-2}t^2+\cdots\bigg)^{(e-\gamma)/d},
$$
thanks to \eqref{eq:xA}. 
Then, note that the constraint $
\sum_{\alpha\le d-1} (d-\alpha)\nu_{\gamma,\alpha}=\mu-\gamma
$ is equivalent to the restriction to degree $t^{\mu-\gamma}$. 
Thus the right side of  \eqref{Magnus_formula} is equal to the right side of  \eqref{eq:gemu_variation}.
\end{proof}

We will prove the following statement, which is equivalent to Theorem \ref{Magnus_thm}.
\begin{theorem}\label{prop:Magnus_formula_equivalence}
Suppose $[F, G] \in  K$. For any direction $w=(u,v) \in W$, let $d=w\text{-}\deg(F_+)$ and $e=w\text{-}\deg(G_+)$.  Assume $d>0$. Write the $w$-homogeneous  decompositions $F=\sum_{i\leq d} F_i$ and $G=\sum_{i\leq e} G_i$. Define 
$$\widetilde{F}=\ff_d+\ff_{d-1}t+\cdots \quad\text{and}\quad \widetilde{G}=G_e+G_{e-1}t+\cdots.$$
Let $r\in\mathbb{Z}_{>0}$ be the largest integer such that $F_d^{1/r}\in K[x,y]$. 
Then there exists a unique sequence of constants $c_0,c_1,\dots,c_{d+e-u-v-1}\in K$ such that $c_0\neq 0$ and 
\begin{equation}\label{Magnus_formula_equivalence}
G_{e-\mu} = \sum_{\gamma=0}^{\mu} c_\gamma [\widetilde{F}^{\frac{e-\gamma}{d}}]_{t^{\mu-\gamma}}
\end{equation}
for every integer $\mu\in\{0,1,...,d+e-u-v-1\}$. 
Moreover, $c_\gamma=0$ if $r(e-\gamma)/d\notin\mathbb{Z}$. 
\end{theorem}

Note that the last condition implies that every nonzero summand appearing on the right side of \eqref{Magnus_formula_equivalence} is in $\mathcal{R}[F_d^{-1/r}]$, so must be a rational function.

In order to prove Theorem~\ref{prop:Magnus_formula_equivalence}, we need the following the lemma (cf. \cite[Propositions 1,2]{NaBa}, \cite[Lemma 22]{ApOn}, \cite[p258]{Magnus1}) and for the readers' convenience we reproduce the proof here.
\begin{lemma}\label{Magnus lemma}
Let $w=(u,v)\in W$. 
Let $R$ be any polynomial ring over $K$, $f\in R$ be a $w$-homogeneous polynomial of degree $d_f>0$, and $g$ be a nonzero $w$-homogeneous function of degree $d_g\in\mathbb{Z}$ in the fractional field of $R$ such that the Jacobian $[g,f]=0$. 
Define $r\in\mathbb{Z}_{>0}$ to be the largest integer such that $h=f^{1/r}$ is a polynomial. 
Then there exists a unique $c\in K\setminus\{0\}$ so that $g=c\cdot h^s$, where $s=rd_g/d_f$ is an integer. 
\end{lemma}
\begin{proof}
If $g\in K$, the statement is trivial where $s=0$, $c=g$.  So for the rest of the proof we assume that $g$ is not a constant.
By Euler's Lemma, $uxf_x+vyf_y=d_ff$, $uxg_x+vyg_y=d_gg$. Then
$$\begin{bmatrix} d_ff\\d_gg\end{bmatrix}=\begin{bmatrix} f_x&f_y\\g_x&g_y\end{bmatrix}\begin{bmatrix} ux\\vy\end{bmatrix},$$ 
so 
$$\begin{bmatrix} d_ffg_y-d_ggf_y\\ -d_ffg_x+d_ggf_x\end{bmatrix}
=\begin{bmatrix} g_y&-f_y\\-g_x&f_x\end{bmatrix}\begin{bmatrix} d_ff\\d_gg\end{bmatrix}
=  \begin{bmatrix} g_y&-f_y\\-g_x&f_x\end{bmatrix}\begin{bmatrix} f_x&f_y\\g_x&g_y\end{bmatrix}\begin{bmatrix} ux\\vy\end{bmatrix}
=  \begin{bmatrix} 0&0\\0&0\end{bmatrix}\begin{bmatrix} ux\\vy\end{bmatrix}
=\begin{bmatrix} 0\\0\end{bmatrix}$$
Thus 
$(g^{d_f}/f^{d_g})_x=d_fg^{d_f-1}g_xf^{-d_g}-d_gf^{-d_g-1}f_xg^{d_f}
=g^{d_f-1}f^{-d_g-1}(d_ffg_x-d_ggf_x)=0$, and similarly 
$(g^{d_f}/f^{d_g})_y=0$. So $g^{d_f}/f^{d_g}=c'$ for some $c'\in K\setminus\{0\}$.

Let $a_1,a_2\in K\setminus\{0\}$, $p_1,\dots,p_n$ be distinct irreducible polynomials, $r_1,\dots,r_n\in\mathbb{Z}_{\ge0}$, $s_1,\dots,s_n\in\mathbb{Z}$, such that we have the prime factoriaztion $f=a_1p_1^{r_1}\cdots p_n^{r_n}$,  $g=a_2p_1^{s_1}\cdots p_n^{s_n}$.  Then $d_fs_i=d_gr_i$ for $1\le i\le n$, and $r={\rm gcd}(r_1,\dots,r_n)$. 
Let $s'={\rm gcd}(s_1,\dots,s_n)>0$. We have $d_fs'={\rm gcd}(d_fs_1,\dots,d_fs_n)={\rm gcd}(d_gr_1,\dots,d_gr_n)=|d_g|r$. So $s=\pm s'$ is an integer, $r_i:s_i=d_f:d_g=r:s$. So the exponent of $p_i$ in the prime factorization of $c=g/h^s$ is $s_i-s(r_i/r)=0$; thus $c$ is a constant.  The uniqueness of $c$ follows from the previous sentences. 
\end{proof}

\begin{proof}[Proof of Theorem \ref{prop:Magnus_formula_equivalence}]
We proceed by induction on $\mu$. 
The base case of $\mu=0$ is $\GG_e=c_0 \ff_d^{e/d}$, which follows from Lemma \ref{Magnus lemma}.
Note that $c_0\neq0$, because otherwise $\GG_e=c_0\ff_d^{e/d}=0$ which contradicts the assumption that $e=w\text{-}\deg(G_+)$.

For the inductive step, assume $\mu>0$. By inductive assumption, $c_0,\dots,c_{\mu-1}$ are uniquely determined.
The assumption $[F,G]\in K$ implies that each positive $w$-degree component of $[F,G]$ is 0. 
Note that if the $w$-degrees of homogeneous rational functions $f$ and $g$ are $i$ and $j$ respectively, then the $w$-degree of $[f,g]$ is $i+j-u-v$. 
{ On the other hand, the  component in $[F,G]$ of $w\text{-degree}=(d+e-u-v-\mu)$ is just $\big[[\widetilde{F},\widetilde{G}]\big]_{t^{\mu}}$. 
Since $\mu<{d+e-u-v}$, we have
$$0=\big[[\widetilde{F},\widetilde{G}]\big]_{t^\mu}
=\big[[\sum_{i\ge0} F_{d-i}t^{i},\sum_{j\ge0} G_{e-j}t^{j}]\big]_{t^\mu}
=\sum_{\stackrel{i,j\ge0}{i+j=\mu}}[F_{d-i},G_{e-j}],
$$
therefore
\begin{equation}\label{eq:GF=FG}
[G_{e-\mu}, F_d]=-[F_d,G_{e-\mu}]=\sum_{\stackrel{i>0,j\ge0}{i+j=\mu}}[F_{d-i},G_{e-j}].
\end{equation}
}

Define
$$H=G_{e-\mu}-\sum_{\gamma=0}^{\mu-1} c_\gamma [\widetilde{F}^{\frac{e-\gamma}{d}}]_{t^{\mu-\gamma}}$$
Note that \eqref{Magnus_formula_equivalence} holds if and only if $H=c_\mu F_d^{\frac{e-\mu}{d}} \ (=c_\mu [\widetilde{F}^{\frac{e-\mu}{d}}]_{t^0})$. 

If $H=0$, then the equation  $H=c_\mu F_d^{\frac{e-\mu}{d}}$ holds exactly when $c_\mu=0$, so the choice of $c_\mu$ is unique. 

Now assume $H\neq0$. It is a homogeneous rational function in $\mathcal{R}[F_d^{-1/r}]$ of $w\text{-deg}=(e-\mu)$ by the inductive hypothesis.
We claim that $[H,F_d]=0$. Indeed,
\begingroup
\allowdisplaybreaks
\begin{align*}
[H,F_d]&=[G_{e-\mu},F_d]-[\sum_{\gamma=0}^{\mu-1} c_\gamma [\widetilde{F}^{\frac{e-\gamma}{d}}]_{t^{\mu-\gamma}},F_d] \quad\text{ (by the definition of $H$)}\\
&=\sum_{\stackrel{i>0,j\ge0}{i+j=\mu}}[F_{d-i},G_{e-j}]-\sum_{\gamma=0}^{\mu-1} [c_\gamma [\widetilde{F}^{\frac{e-\gamma}{d}}]_{t^{\mu-\gamma}},F_d] \quad\text{ (by \eqref{eq:GF=FG})}\\
&=
\sum_{j=0}^{\mu-1}[F_{d-\mu+j},
\sum_{\gamma=0}^{j} c_\gamma [\widetilde{F}^{\frac{e-\gamma}{d}}]_{t^{j-\gamma}}]
-\sum_{\gamma=0}^{\mu-1} [c_\gamma [\widetilde{F}^{\frac{e-\gamma}{d}}]_{t^{\mu-\gamma}},F_d]\quad\text{ (by the inductive hypothesis)}\\
&=
\sum_{\gamma=0}^{\mu-1}c_\gamma 
\Big(
\sum_{j=\gamma}^{\mu-1}\big[F_{d-\mu+j},[\widetilde{F}^{\frac{e-\gamma}{d}}]_{t^{j-\gamma}}\big]
-\big[[\widetilde{F}^{\frac{e-\gamma}{d}}]_{t^{\mu-\gamma}},F_d \big]
\Big)
\\
&=
\sum_{\gamma=0}^{\mu-1}c_\gamma 
\Big(
\sum_{j=\gamma}^{\mu}\big[F_{d-\mu+j},[\widetilde{F}^{\frac{e-\gamma}{d}}]_{t^{j-\gamma}}\big]
\Big)
=
\sum_{\gamma=0}^{\mu-1}c_\gamma 
\Big(
\sum_{j=\gamma}^{\mu}\big[[\widetilde{F}]_{t^{\mu-j}},[\widetilde{F}^{\frac{e-\gamma}{d}}]_{t^{j-\gamma}}\big]
\Big)
\\
&=
\sum_{\gamma=0}^{\mu-1}c_\gamma 
\Big(
\big[[\widetilde{F},\widetilde{F}^{\frac{e-\gamma}{d}}]\big]_{t^{\mu-\gamma}}
\Big)=0\quad\text{ (because $[\widetilde{F},\widetilde{F}^{\frac{e-\gamma}{d}}]=0$)}
\\
\end{align*}
\endgroup
This equality together with Lemma \ref{Magnus lemma} (with $g=H, f=F_d, h=F_d^{1/r}$) implies that there is a unique element $c_\mu\in K\setminus\{0\}$ such that 
$$H=c_\mu h^s=c_\mu F_d^{\deg H/\deg F_d}=c_\mu F_d^{\frac{e-\mu}{d}},$$
where $s=r(e-\mu)/d\in\mathbb{Z}$. In other words, $c_\mu=0$ if $r(e-\mu)/d\notin\mathbb{Z}$.   
\end{proof}

For any real numbers $r_1\le r_2$, we use the usual notation for a closed interval $[r_1,r_2]:=\{x\in \mathbb{R} \ : \ r_1\le x\le r_2\}$, and introduce the notation $[r_1,r_2]_{\mathbb{Z}}:=[r_1,r_2]\cap \mathbb{Z}$. 
For a  line segment $\overline{AB}\subset \mathbb{R}^2$ whose endpoints are both in $\mathbb{Z}^2\subset  \mathbb{R}^2$, we define the length ${\rm len}(\overline{AB})\in \mathbb{Z}_{\ge 0}$ to be one less than the number of lattice points on $\overline{AB}$. 
For any direction $w=(u,v)\in W$ and for any $w$-homogeneous Laurent polynomial $h\in K[x^{\pm1}, y^{\pm1}]$, we define ${\rm len}(h)$ to be the length of $N(h)$; that is, if $h= a_0x^by^c+a_1x^{b+v}y^{c-u}+a_2x^{b+2v}y^{c-2u}+\cdots+a_lx^{b+lv}y^{c-lu}$ with $a_0\neq0$ and $a_l\neq0$, then ${\rm len}(h)=l$.  

In the following statement, for each polynomial $F,P$ and $R$, we fix  $w=(u,v)\in W$ and write the $w$-homogeneous degree decompositions $F=\sum_i F_i$, $P=\sum_i P_i$, and $R=\sum_i R_i$.

\begin{proposition}\label{application_of_Magnus}
Let $k\in\mathbb{Z}_{>0}$ and assume that $(\ff,\GG)$ has generic boundaries with $a=2$ and $b=2k+1$. Denote $d:=w\emph{-deg}(F_+)$ and $e:=w\emph{-deg}(G_+)$. Let $P\in\mathcal{R}$  such that $w\emph{-deg}(P_+)=m=d/2$ and $\ff_d=P_m^2$.   Let $R=F-P^2$ and $h\in [1,2m-1]_{\mathbb{Z}}$. 
If $R_{d-\ell}=0$ for all $\ell<h$, then $P_m^{-1}R_{d-h}\in K[x^{\pm 1},y^{\pm 1}]$.
\end{proposition}

\begin{proof}
We assume $P_m$ is not a monomial (thus $F_d$ is not a monomial) since the statement is trivial otherwise.
Let 
$$\aligned
&\widetilde{F}=\ff_d+\ff_{d-1}t+\ff_{d-2}t^2+\cdots \in \mathcal{R}[[t]],\\
&Q=P_m+P_{m-1}t+\cdots+P_{m-h}t^{h} \in \mathcal{R}[[t]],\\
&\text{and }T=\widetilde{F}-Q^2.
\endaligned$$ 
For each positive integer $z$, let $\mathcal{O}(t^{z})$ denote an element of the form $\sum_{i\ge z} f_i t^i$ in $\mathcal{R}[[t]]$, where each $f_i$ is a polynomial in $\mathcal{R}$ which we do not have to care about. Since $R_{d-\ell}=0$ for $\ell<h$, we get
$$T = \left(\ff_{d-h}-\sum_{j=0}^{h}P_{m-j}P_{m-h+j}\right)t^h+ \mathcal{O}(t^{h+1})=R_{d-h}t^h+ \mathcal{O}(t^{h+1}).$$

We will apply  Theorem~\ref{prop:Magnus_formula_equivalence} to the case of $\mu=h(k+1)$. For that purpose, we need to check that $h(k+1)\le d+e-u-v-1$. Observe:  $$\aligned d+e-u-v-1-h(k+1)&\ge 2m+(2k+1)m-u-v-1-(2m-1)(k+1)\\
&=m-u-v+k\ge m-u-v +1=m+1 - (w\text{-deg}(xy))\ge 0,\endaligned$$
where the last inequality  holds because $m-(w\text{-deg}(xy))\ge 0$ if the lattice point $(1,1)$ is contained in $N(P)$, or $w\text{-deg}(xy)=m+1$ otherwise. Indeed, if $(1,1)$ is not contained in $N(P)$, then $N(P)$ is either the triangle with vertices $(0,0),(c,0),(0,1)$  or the triangle with vertices  $(0,0),(1,0),(0,c)$ for some $c\in\mathbb{Z}_{>0}$. Without loss of generality we assume the former. Since $P_m$ is not a monomial, both the points $(c,0)$ and $(0,1)$ must lie in the support of $P_m$. So $w=(1,c)$, $m=c$, and thus $w\text{-deg}(xy)=1+c=m+1$.


Since $(\ff,\GG)$ has generic boundaries, Theorem~\ref{prop:Magnus_formula_equivalence} gives the following:
$$\aligned
\GG_{e-h(k+1)} 
&= \sum_{r=0}^{\lfloor h(k+1)/m\rfloor} c_{rm}
\bigg[\widetilde{F}^{k+1/2-r/2}\bigg]_{t^{h(k+1)-rm}}\\
&= \sum_{r=0}^{\lfloor h(k+1)/m\rfloor} c_{rm}
\bigg[(Q^2+T)^{k+1/2-r/2}\bigg]_{t^{h(k+1)-rm}}\\
&= \sum_{r=0}^{\lfloor h(k+1)/m\rfloor} c_{rm} 
\bigg[\sum_{i=0}^\infty\binom{k+1/2-r/2}{i}Q^{2k+1-r-2i}T^i\bigg]_{t^{h(k+1)-rm}}.
\endaligned
$$
Note that $T=\mathcal{O}(t^{h})$, which implies $T^i=\mathcal{O}(t^{hi})$. Then it is enough to look at $r$ and $i$ such that
$hi\le h(k+1)-rm$, or equivalently $i\le k+1-rm/h$. 
Then the exponent of $Q$ satisfies
$$2k+1-r-2i\ge 2k+1-r-2(k+1-rm/h)=2rm/h-r-1\ge 2rm/d-r-1=r-r-1=-1.$$
Here  the first ``$\ge$'' becomes ``$=$'' only when $i=k+1-rm/h$, and  the second ``$\ge$'' becomes ``$=$'' only when $r=0$. Since the exponent of $Q$ is an integer, it is always nonnegative except when ``$r=0$ and $i=k+1$'' (in which case the exponent is $-1$). Since $\GG_{e-h(k+1)}$ and $[Q^{2k+1-r-2i}T^i]_{t^{h(k+1)-rm}}$ are polynomials in $K[x,y]$ whenever $2k+1-r-2i\ge 0$,   the following must also be a polynomial in $K[x,y]$:
$$\aligned
&c_0\binom{k+1/2}{k+1}[Q^{-1}T^{k+1}]_{t^{h(k+1)}}\\
&=c_0\binom{k+1/2}{k+1}\bigg[(P_m+\mathcal{O}(t))^{-1} \Big(R_{d-h}t^h+\mathcal{O}(t^{h+1})\Big)^{k+1}\bigg]_{t^{h(k+1)}}\\
&=c_0\binom{k+1/2}{k+1}P_m^{-1} (R_{d-h})^{k+1}.
\endaligned
$$
Since $c_0\neq 0$ and $\binom{k+1/2}{k+1}\neq 0$, we get that  $P_m^{-1} (R_{d-h})^{k+1}$ is a polynomial. Since $(F, G)$ has generic boundaries, the polynomial $P_m$ is not divisible by the square of any polynomial in $\mathcal{T}$. This implies $P_m^{-1} R_{d-h}\in K[x^{\pm1},y^{\pm1}]$, because $K$ is algebraically closed. 
 \end{proof}

\begin{corollary}\label{application2_of_Magnus}
The same hypotheses as above. If $\text{len}(R_{d-h})<\text{len}(P_m)$ then $R_{d-h}= 0$.
\end{corollary}
 \begin{proof}
If $R_{d-h}\neq 0$ then $\text{len}(R_{d-h})\ge \text{len}(P_m)$, because $P_m^{-1} R_{d-h}\in K[x^{\pm1},y^{\pm1}]$. 
  \end{proof}

The following lemma is elementary but makes Proposition~\ref{application_of_Magnus} useful.
Let $\ff \in\mathcal{R}$ be a polynomial with a nonzero constant term such that all vertices of $N(F)$ are in $(2\mathbb{Z})^2$, and let $C$ be any vertex  of $N(\ff)$ other than the point of origin $O$.  Let $\mathcal{N}'=\frac{1}{2} N(\ff)$ which is defined at the end of Section~\ref{section_intro}, and $\mathcal{N}''=\mathcal{N}'+\frac{1}{2}\overrightarrow{OC}$. (For example, see the polygons shown in Figure \ref{fig:newton arbitrary}.)

\begin{lemma}\label{lem:c=p^2 most generalized}
Let $\ff=\sum_{i,j}\lambda_{ij}x^iy^j$.   There exists a polynomial $P=\sum_{(i,j)\in \mathcal{N}'} p_{ij}x^iy^j$,  unique up to a sign, such that $\emph{supp}(\ff-P^2)\cap \mathcal{N}''=\emptyset$.
\end{lemma}
\begin{proof}
First consider the case that $N(\ff)$ is a rectangle $[0,2m']\times[0,2m]$ for $m',m\in \mathbb{Z}_{>0}$. In particular, $\lambda_{2m',2m}\neq0$. See Figure \ref{fig:newton_rectangle}. Let $C=(2m',2m)$. 
The required property gives a system of $(m+1)(m'+1)$ quadratic equations with $(m+1)(m'+1)$ variables $p_{ij}$. We can solve $p_{ij}$ recursively, in the following  ``graded lex order'': $p_{m',m}>p_{m',m-1}>p_{m'-1,m}>p_{m',m-2}>p_{m'-1,m-1}>p_{m'-2,m}>\cdots$.  Namely, first use $\lambda_{2m',2m}-p_{m',m}^2=0$ to determine $p_{m',m}\neq0$ up to a sign; next use $\lambda_{2m',2m-1}-2p_{m',m}p_{m',m-1}=0$ to uniquely determine $p_{m',m-1}$, etc. 

Even if $N(\ff)$ is arbitrary, we can still solve $p_{ij}$ recursively with respect to an appropriate order in the same way as follows. 
 Let $L_c=\{(x,y)\ | \alpha x+\beta y=c\}$ (for some $c>0$) be a line with irrational slope that passes $C$ and intersect with $N(F)$ only at $C$. Then $N(F)$ lies in the half plane $\alpha x+\beta y\le c$. Arrange points $\{z_i=(x_i,y_i)\}_{1\le i\le n}$ in $\mathcal{N}'$ such that $\alpha x_1+\beta y_1>\alpha x_2+\beta y_2>\cdots>\alpha x_n+\beta y_n$. Then $z_1=\frac{1}{2}C$. For Denote ${\bf x}^{z_i}=x^{x_i}y^{y_i}$. Then in our new notation, $P=\sum_{z_i} p_{z_i}{\bf x}^{z_i}$. We claim that we can solve $p_{z_1},p_{z_2},\dots, p_{z_n}$ recursively. First use $\lambda_{2z_1}-p_{z_1}^2=0$ to determine $p_{z_1}$ up to a sign. If $p_{z_1},\dots,p_{z_{k-1}}$ are determined, then using the equaltion 
$$\lambda_{z_1+z_k}=\sum_{z_i+z_j=z_1+z_k} p_{z_i}p_{z_j}=2p_{z_1}p_{z_k}+\sum_{z_i+z_j=z_1+z_k,1<i,j<k} p_{z_i}p_{z_j}$$
we can uniquely determine $p_{z_k}$. Since $\mathcal{N}''=\{z_1+z_k\ | \  1\le k\le n\}$, we have found a unique $P$ (up to a sign) such that  $\text{supp}(\ff-P^2)\cap \mathcal{N}''=\emptyset$. 
\end{proof}

\section{Proof of Theorem~\ref{main_thm} for the case where $N(\ff)$ is a rectangle}\label{rec}
Let $\ff=\sum_{i,j}\lambda_{ij}x^iy^j$. 
In this section we prove Theorem~\ref{main_thm}, assuming that $N(\ff)$ is a rectangle $[0,2m']\times[0,2m]$ for $m',m\in \mathbb{Z}_{>0}$. In particular, $\lambda_{2m',2m}\neq0$. See Figure \ref{fig:newton_rectangle}.
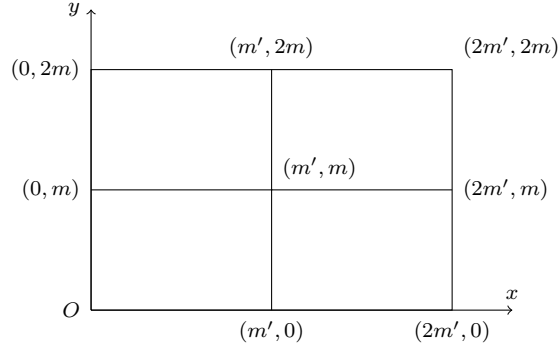
\begin{figure}[h]
\begin{center}
\begin{tikzpicture}[scale=0.80]
\usetikzlibrary{patterns}
\draw (0,0)--(0,4)--(6,4)--(6,0)--(0,0);
\draw (0,2)--(6,2) (3,0)--(3,4);
\draw (0,0) node[anchor=east] {\tiny$O$};
\draw (0,4) node[anchor=east] {\tiny$(0,2m)$};
\draw (6,4) node[anchor=south west] {\tiny$(2m',2m)$};
\draw (6,0) node[anchor=north] {\tiny$(2m',0)$};
\draw (3,4) node[anchor=south] {\tiny$(m',2m)$};
\draw (6,2) node[anchor=west] {\tiny$(2m',m)$};
\draw (0,2) node[anchor=east] {\tiny$(0,m)$};
\draw (3,2) node[anchor=south west] {\tiny$(m',m)$};
\draw (3,0) node[anchor=north] {\tiny$(m',0)$};
\draw[->] (0,0) -- (7,0)
node[above] {\tiny $x$};
\draw[->] (0,0) -- (0,5)
node[left] {\tiny $y$};
\end{tikzpicture}
\end{center}
\caption{The case where $N(\ff)$ is a rectangle }
\label{fig:newton_rectangle}
\end{figure}

Let $P$ be a polynomial given by Lemma~\ref{lem:c=p^2 most generalized}, and let $R=F-P^2$. 
Write the $w$-homogeneous degree decomposition $R=\sum_i R_i$ for $w=(0,1)\in W$, and  the $w'$-homogeneous degree decomposition $R=\sum_i R'_i$ for $w'=(1,0)\in W$. 

\begin{proposition}\label{sum_PP} Suppose that $(\ff,\GG)$ has generic boundaries with $a=2$.  Let $d'=2m'$, $d=2m$ and $e'=(2k+1)m'$, $e=(2k+1)m$ for some positive integer $k$. Then we have $R_{d-h}=0$  for $h\in [0,d-1]_{\mathbb{Z}}$, and $R'_{d'-h}=0$ for $h\in [0,d'-1]_{\mathbb{Z}}$. 
\end{proposition}
\begin{proof}
We will use induction on $h$.  The base case of $h=0$ is well known (for instance, see \cite{A,ApOn,NaBa}), and also follows from Theorem~\ref{Magnus_thm} for $\mu=0$.

Let $h\in [1,m]_{\mathbb{Z}}$, and assume the inductive hypothesis that $R_{d-\ell}=0$ for $\ell<h$.  Lemma~\ref{lem:c=p^2 most generalized} implies that 
$\text{len}(R_{d-h})<m'=\text{len}(P_+)$. Hence $R_{d-h}=0$ by Corollary~\ref{application2_of_Magnus}.

Applying the same argument, we also get $R'_{d'-h}=0$  for $h\in [0,m']_{\mathbb{Z}}$. Then the support of $R$ is contained in $[0,m'-1]\times[0,m-1]$, which in turn implies that $R_{d-h}=0$  for $h\in [m+1,d-1]_{\mathbb{Z}}$, and 
$R'_{d'-h}=0$  for $h\in [m'+1,d'-1]_{\mathbb{Z}}$. \end{proof}

In light of this proposition, we can show the following. 
\begin{corollary}\label{sum_PP_imply} Suppose that $(\ff,\GG)$ has generic boundaries with $a=2$. If $N(\ff)$ is a rectangle, then $F=P^2+u_0$ for some constant $u_0\in K$. In particular, $[\ff,\GG]=0$.
\end{corollary}
\begin{proof}
Proposition~\ref{sum_PP} implies that  the support of $R$ is either the origin or empty, so $R$ is equal to a constant, say $u_0\in K$. That is,
$\ff=P^2+u_0$. This gives $[\ff,\GG]=[P^2,\GG]=2P[P,\GG]\in K$.  So $P\in K$, which gives  $[\ff,\GG]=0$. 
\end{proof}

\begin{example}
Assume that $a=2,b=3$, $N(F)$ is the $2\times2$ square, $N(\GG)$ is the $3\times3$ square, and $\ff$ takes the form
$$ F= \lambda_{2,2}x^2y^2 + \lambda_{2,1}x^2y + \lambda_{1,2}xy^2 + \lambda_{1,1}xy + \lambda_{2,0}x^2 + \lambda_{0,2}y^2 + \lambda_{1,0}x + \lambda_{0,1}y + \lambda_{0,0},$$
where $\lambda_{i,j}\in K$ and $\lambda_{2,2}\lambda_{2,0}\lambda_{0,2}\neq 0$. See Figure \ref{fig:ex_1}. 
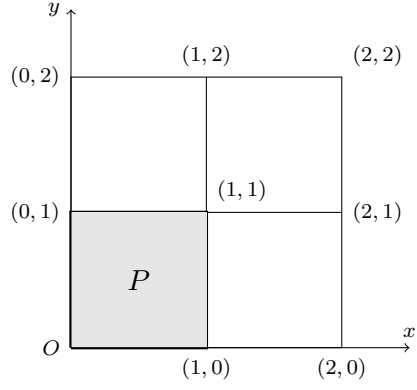
\begin{figure}[h]
\begin{center}
\begin{tikzpicture}[scale=0.90]
\usetikzlibrary{patterns}
\draw (0,0)--(0,4)--(4,4)--(4,0)--(0,0);
\draw (0,2)--(4,2) (2,0)--(2,4);
\draw[line width = 1.25pt] (0,0)--(0,2)--(2,2)--(2,0)--(0,0);
\fill[black!10](0,0)--(0,2)--(2,2)--(2,0)--(0,0);
\draw (1,1) node {\small$P$};
\draw (0,0) node[anchor=east] {\tiny$O$};
\draw (0,4) node[anchor=east] {\tiny$(0,2)$};
\draw (4,4) node[anchor=south west] {\tiny$(2,2)$};
\draw (4,0) node[anchor=north] {\tiny$(2,0)$};
\draw (2,4) node[anchor=south] {\tiny$(1,2)$};
\draw (4,2) node[anchor=west] {\tiny$(2,1)$};
\draw (0,2) node[anchor=east] {\tiny$(0,1)$};
\draw (2,2) node[anchor=south west] {\tiny$(1,1)$};
\draw (2,0) node[anchor=north] {\tiny$(1,0)$};
\draw[->] (0,0) -- (5,0)
node[above] {\tiny $x$};
\draw[->] (0,0) -- (0,5)
node[left] {\tiny $y$};
\end{tikzpicture}
\end{center}
\caption{The Newton polygons associated with $F$ and $P$, with $P$ in bold}
\label{fig:ex_1}
\end{figure}

By Lemma~\ref{lem:c=p^2 most generalized}, there exists a polynomial $P = P_{1,1}xy + P_{1,0}x + P_{0,1}y + P_{0,0}$ such that $\text{supp}(F - P^2)$ is contained in $\{(0,2),(0,1),(0,0),(1,0),(2,0)\}$. We  immediately deduce that:
\begin{align*}
    \lambda_{2,2} & = P_{1,1}^2, \\
   \lambda_{2,1} & = 2P_{1,1}P_{1,0}, \\
    \lambda_{1,2} & = 2P_{1,1}P_{0,1}, \\
    \lambda_{1,1} & = 2P_{1,0}P_{0,1} + 2P_{1,1}P_{0,0}. \\
\end{align*}
Suppose that $[\ff,\GG]\in K$. Applying  \eqref{Magnus_formula} to $w=(0, 1)$ and $\mu=0$, we have
$$\lambda_{2,2}x^2y^2 + \lambda_{1, 2}xy^2 + \lambda_{0,2}y^2 = P_{1,1}^2 x^2y^2 + 2P_{1,1}P_{0,1} xy^2 + \lambda_{0,2}y^2 = (P_{1,1}' xy + P_{0,1}'y)^2$$  for some constants $P_{1,1}' ,P_{0,1}'\in K$, which 
    implies $\lambda_{0,2} = P_{0,1}^2$ and $(P_{1,1} xy + P_{0,1}y)^2=(P_{1,1}' xy + P_{0,1}'y)^2$. 

Applying \eqref{Magnus_formula} to $w=(0, 1)$ and $\mu=2$, we get that
$\lambda_{2,1}x^2y +  \lambda_{1,1}xy +  \lambda_{0,1}y$ is divisible by  $P_{1,1}' xy + P_{0,1}'y$, hence divisible by $P_{1,1} xy + P_{0,1}y$. This means that
$$\aligned
\lambda_{2,1}x^2y +  \lambda_{1,1}xy +  \lambda_{0,1}y &= 2P_{1,1}P_{1,0}x^2y +  (2P_{1,0}P_{0,1} + 2P_{1,1}P_{0,0})xy +  \lambda_{0,1}y \\
&= 2(P_{1,1}xy + P_{0,1}y)(P_{1,0}'x + P_{0,0}') \endaligned$$
 for some constants $P_{1,0}' ,P_{0,0}'\in K$, which 
    implies  $\lambda_{0,1} = 2P_{0,1}P_{0,0}$. 

Similarly, applying Magnus' formula to $w=(1, 0)$, we also have
$\lambda_{2,0} = P_{1,0}^2$
and
$\lambda_{1,0} = 2P_{1,0}P_{0,0}$.
From here, it follows that $F - P^2$ is a constant.
\end{example}

\section{Construction of broken lines and proof of Theorem~\ref{main_thm}}
In this section, we prove Theorem~\ref{main_thm}. Besides from using 
Proposition~\ref{application_of_Magnus}, 
Corollary~\ref{application2_of_Magnus} and
Lemma~\ref{lem:c=p^2 most generalized}, the rest  is a purely combinatorial analysis on subsets of $\mathbb{R}^2$ and $\mathbb{Z}^2$. Suppose that $(\ff,\GG)$ has generic boundaries with $a=2$. 

Denote the point of origin by $O=A_0=B_0$. Now we extend the result from Section~\ref{rec} to the general case that  $N(\ff)$ is of arbitrary shape. We say that a vertex $C=(c_x,c_y)\in N(F)$ is  \emph{northeastern} if $(v_x-c_x,v_y-c_y)\not\in\mathbb{Z}_{\ge0}^2$  for any other vertex $V=(v_x,v_y)\in N(F)$. Observe that a northeastern vertex exists.

Let $N(\ff)=OA_1\cdots A_{\alpha-1}CB_{\beta-1}\cdots B_1$ where $C$ is  a northeastern vertex of $N(\ff)$.  We denote $A_\alpha=B_\beta=C$, $C'=\frac{1}{2}C$.  Since $(1,0), (0,1)\in N(F)$, we must have that $A_1$ lies on the $y$-axis and  $B_1$ lies on the $x$-axis. Without loss of generality, assume $\alpha>1$ (but we allow $\beta=1$).

Let $\mathcal{N}'=\frac{1}{2} N(\ff)$ and $\mathcal{N}''=\mathcal{N}'+\frac{1}{2}\overrightarrow{OC}$ be the polygons shown in Figure \ref{fig:newton arbitrary}.

\subsection{Construction of parallelograms associated with $N(F)$}

For $0\le i\le j\le \alpha$, define $A_{ij}=\frac{1}{2}(A_i+A_j)\in \mathbb{R}^2$. In particular, $A_{i,i}=A_i$. 

We define  the parallelogram $\mathcal{P}_{ij}$ ($1\le i<j\le \alpha$) by its four  vertices $A_{i-1,j-1}$, $A_{i-1,j}$, $A_{i,j-1}$, $A_{i,j}$. For convenience, we call the line segments $\overline{A_{i-1,j-1}A_{i,j-1}}$, $\overline{A_{i,j-1}A_{ij}}$, $\overline{A_{ij}A_{i-1,j}}$, $\overline{A_{i-1,j}A_{i-1,j-1}}$ the west, north, east, south edges of $\mathcal{P}_{ij}$, respectively. 
Similarly we also define the parallelogram $\mathcal{P}'_{ij}$ ($1\le i<j\le \beta$). See Figure \ref{fig:newton arbitrary}.

\begin{figure}[h]
\begin{center}
\begin{tikzpicture}[scale=0.60]
\usetikzlibrary{patterns}
\draw (0,0)--(0,2)--(2,6)--(6,8)--(10,6)--(10,4)--(6,0);
\draw (0,0)--(0,1)--(1,3)--(3,4)--(5,3)--(5,2)--(3,0);
\draw (5,3)--(5,4)--(6,6)--(8,7)--(10,6)--(10,5)--(8,3)--(5,3);
\fill [black!10]  (0,0)--(0,1)--(1,3)--(3,4)--(5,3)--(5,2)--(3,0);
\fill [black!10]  (5,3)--(5,4)--(6,6)--(8,7)--(10,6)--(10,5)--(8,3)--(5,3);
\draw (1,3)--(1,4)--(3,5)--(5,4) (3,4)--(3,5)--(4,7)--(6,6);
\draw (8,3)--(8,2)--(5,2);
\draw[dotted] (0,0)--(10,6);
\draw (0,0) node[anchor=east] {\tiny$O$};
\draw (0,2) node[anchor=east] {\tiny$A_1$};
\draw (2,6) node[anchor=south east] {\tiny$A_2$};
\draw (6,8) node[anchor=south] {\tiny$A_3$};
\draw (10,6) node[anchor=west] {\tiny$C=A_4=B_3$};
\draw (5,2.9) node[anchor=south west] {\tiny$C'=A_{04}$};
\draw (10,4) node[anchor=west] {\tiny$B_2$};
\draw (6,0) node[anchor=north] {\tiny$B_1$};
\draw(7.5,5) node {\tiny$\mathcal{N}''$};
\draw(2.5,2) node {\tiny$\mathcal{N}'$};
\draw(0.6,1.8) node[anchor=south] {\tiny$\mathcal{P}_{12}$};
\draw(2,3.6) node[anchor=south] {\tiny$\mathcal{P}_{13}$};
\draw(4,3.7) node[anchor=south] {\tiny$\mathcal{P}_{14}$};
\draw(2.5, 5) node[anchor=south] {\tiny$\mathcal{P}_{23}$};
\draw(4.5, 5) node[anchor=south] {\tiny$\mathcal{P}_{24}$};
\draw(6,6.5) node[anchor=south] {\tiny$\mathcal{P}_{34}$};
\draw(5.5,0.7) node[anchor=south] {\tiny$\mathcal{P}'_{12}$};
\draw(6.5,2.1) node[anchor=south] {\tiny$\mathcal{P}'_{13}$};
\draw(8.9,3) node[anchor=south] {\tiny$\mathcal{P}'_{23}$};
\draw[->] (0,0) -- (10,0)
node[above] {\tiny $x$};
\draw[->] (0,0) -- (0,7)
node[left] {\tiny $y$};
\draw (0,1) node[anchor=east] {\tiny$A_{01}$};
\draw (.8,2.9) node[anchor=west] {\tiny$A_{02}$};
\draw (3,4) node[anchor=north] {\tiny$A_{03}$};
\draw (5,4) node[anchor=west] {\tiny$A_{14}$};
\draw (5.9,6) node[anchor=west] {\tiny$A_{24}$};
\draw (8,7) node[anchor=north] {\tiny$A_{34}$};
\begin{scope}[shift={(14,0)}]
\usetikzlibrary{patterns}
\draw[black!20]  (0,0)--(0,2)--(2,6)--(6,8)--(10,6)--(10,4)--(6,0);
\draw[black!20]  (0,0)--(0,1)--(1,3)--(3,4)--(5,3)--(5,2)--(3,0);
\draw[black!20]  (5,3)--(5,4)--(6,6)--(8,7)--(10,6)--(10,5)--(8,3)--(5,3);
\fill [black!10]  (0,0)--(0,1)--(1,3)--(3,4)--(5,3)--(5,2)--(3,0);
\fill [black!10]  (5,3)--(5,4)--(6,6)--(8,7)--(10,6)--(10,5)--(8,3)--(5,3);
\draw[black!10]  (1,3)--(1,4)--(3,5)--(5,4) (3,4)--(3,5)--(4,7)--(6,6);
\draw[black!10]  (8,3)--(8,2)--(5,2);
\draw[dotted] (0,0)--(10,6) (0,0)--(1,3) (0,0)--(3,4);
\draw[red] (5.4,7.7)--(9.59,5.75); \fill[red] (9.59,5.75) circle[radius=2pt];
\draw[red] (2,6)--(4,7)--(8.18,4.9); \fill[red] (8.18,4.9) circle[radius=2pt];
\draw[red] (1.5,5)--(3.5,6)--(5.5,5)--(7.05,4.23); \fill[red] (7.05, 4.23) circle[radius=2pt];
\draw[red] (0,2)--(1,4)--(3,5)--(5,4)--(5.9,3.55);\fill[red] (5.9,3.55) circle[radius=2pt];
\draw[red] (0,1.6)--(1,3.6)--(3,4.6)--(5,3.6)--(5.5,3.33);\fill[red] (5.5, 3.33) circle[radius=2pt];
\draw[red] (0,1)--(1,3)--(3,4)--(5,3);\fill[red] (5,3) circle[radius=2pt];
\draw[red] (0,0.6)--(.6,1.8)--(1.8,2.4)--(3,1.8);\fill[red] (3,1.8) circle[radius=2pt];
\draw (0,0) node[anchor=east] {\tiny$O$};
\draw (0,2) node[anchor=east] {\tiny$A_1$};
\draw (2,6) node[anchor=south east] {\tiny$A_2$};
\draw (6,8) node[anchor=south] {\tiny$A_3$};
\draw (10,6) node[anchor=west] {\tiny$C$};
\draw (10,4) node[anchor=west] {\tiny$B_2$};
\draw (6,0) node[anchor=north] {\tiny$B_1$};

\draw[->] (0,0) -- (10,0)
node[above] {\tiny $x$};
\draw[->] (0,0) -- (0,7)
node[left] {\tiny $y$};
\end{scope}
\end{tikzpicture}
\end{center}
\caption{The case where $N(\ff)$ is arbitrary, $\alpha=4,\beta=3$. Left: regions $\mathcal{N}'$, $\mathcal{N}''$, $\mathcal{P}_{ij}$; Right: various broken lines $T_D$ where the red dots are the various positions of the point $D$}
\label{fig:newton arbitrary}
\end{figure}
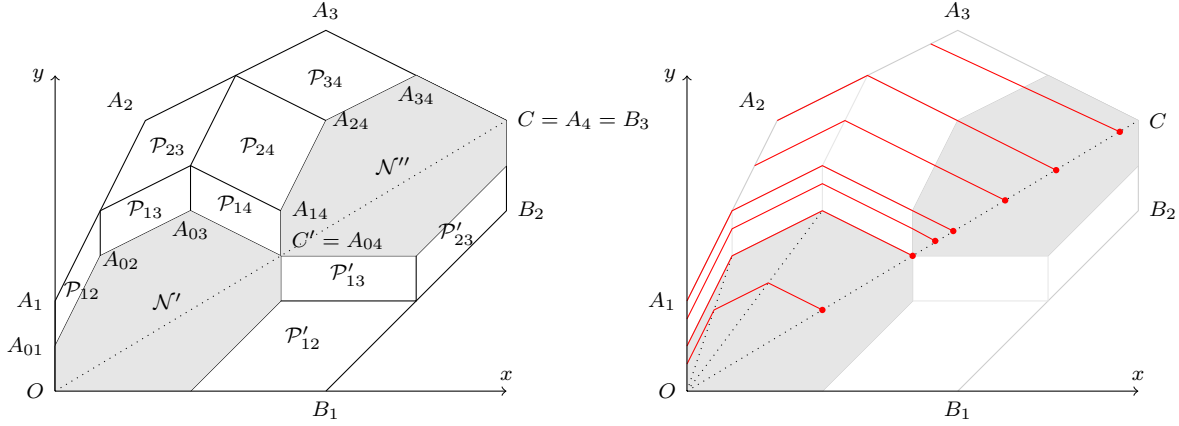

One can verify that \\
\noindent $\bullet$ $\mathcal{P}_{ij}$ is indeed a parallelogram,\\
\noindent $\bullet$  the lengths of the edges of $\mathcal{P}_{ij}$ are equal to $\frac{1}{2}A_{i-1}A_i$ and $\frac{1}{2}A_{j-1}A_j$, \\
\noindent $\bullet$  $\mathcal{P}_{i,j}$ shares edges with $\mathcal{P}_{i,j\pm1}$ and $\mathcal{P}_{i\pm1,j}$ (if the latter are defined). 

We make the following claim.

\begin{lemma}
(a) The union of the parallelograms $\mathcal{P}_{ij}$ is the following (closed and non-convex) polygon:
$$\mathcal{P}:=A_{01}A_{02}\cdots A_{0\alpha}A_{1\alpha}A_{2\alpha}\cdots A_{\alpha-1,\alpha}A_{\alpha-1}A_{\alpha-2}\cdots A_1$$

\noindent (b) These parallelograms do not overlap with each other. More precisely, $\mathcal{P}_{ij}\cap\mathcal{P}_{i,j+1}=\overline{A_{i-1,j}A_{ij}}$,  $\mathcal{P}_{ij}\cap\mathcal{P}_{i+1,j}=\overline{A_{i,j-1}A_{ij}}$, $\mathcal{P}_{ij}\cap \mathcal{P}_{i+1,j+1}=A_{ij}$, $\mathcal{P}_{ij}\cap \mathcal{P}_{i+1,j-1}=A_{i,j-1}$, and $\mathcal{P}_{ij}\cap \mathcal{P}_{i',j'}=\emptyset$ if $|i-i'|>1$ or $|j-j'|>1$.

\end{lemma}
\begin{proof}
(a) Given any point $r\in \mathcal{P}$, we assert that $r$ is in some $\mathcal{P}_{ij}$.

Let ${\bf v}_i$ be the vector $\frac{1}{2}\overrightarrow{A_{i-1}A_i}$ for $1\le i\le \alpha$. Then 
$A_{ij}=O+\sum_{k=1}^{i}{\bf v}_k+\sum_{k=1}^{j}{\bf v}_k$, and 
the parallelogram $\mathcal{P}_{ij}=\{A_{i-1,j-1}+s{\bf v}_i+t{\bf v}_j|0\le s,t\le 1\}$. 
The assumption that $C$ is a northeastern vertex of $N(F)$ has the following consequences: $\mathcal{N}'\cap \mathcal{N}''=\{C'\}$, and ${\bf v}_1,\dots, {\bf v}_\alpha$ are in clockwise order and (strictly) in the same half plane $y+\lambda x>0$ for a sufficiently large constant $\lambda\gg 0$. 

One can see that $\mathcal{P}=L_1+L_2$, the Minkowski sum of the following two broken lines 
$$L_1:=OA_{01}A_{02}\cdots A_{0,\alpha-1} \quad \text{ and } \quad L_2:=A_{01}A_{02}\cdots A_{0,\alpha-1}A_{0,\alpha}.$$
It can be visualized as follows: as a point $p$ moves along $L_1$, the broken line $p+L_2$ sweeps out the region $\mathcal{P}$. As a consequence, any point $r\in\mathcal{P}$ is the sum of a point $p\in L_1$ and a point $q\in L_2$. Without loss of generality we may assume that  $p$ lies in $\overline{A_{0,i-1}A_{0,i}}$ and $q$ lies in $\overline{A_{0,j-1}A_{0,j}}$ for some $i\le j$. 
Then
$p=O+\sum_{k=1}^{i-1}{\bf v}_k+s{\bf v}_i$ and $q=O+\sum_{k=1}^{j-1}{\bf v}_k+t{\bf v}_j$. We consider three cases:

\begin{figure}[h]
\begin{center}
\begin{tikzpicture}[scale=0.60]
\usetikzlibrary{patterns}
\draw[black!20]  (0,0)--(0,2)--(2,6)--(6,8)--(10,6)--(10,4)--(6,0);
\draw[black!20]  (0,0)--(0,1)--(1,3)--(3,4)--(5,3)--(5,2)--(3,0);
\draw[black!20]  (5,3)--(5,4)--(6,6)--(8,7)--(10,6)--(10,5)--(8,3)--(5,3);
\fill [black!10]  (0,0)--(0,1)--(1,3)--(3,4)--(5,3)--(5,2)--(3,0);
\fill [black!10]  (5,3)--(5,4)--(6,6)--(8,7)--(10,6)--(10,5)--(8,3)--(5,3);
\draw[black!10]  (1,3)--(1,4)--(3,5)--(5,4) (3,4)--(3,5)--(4,7)--(6,6);
\draw[black!10]  (8,3)--(8,2)--(5,2);
\draw[dotted] (0,0)--(10,6) (0,0)--(1,3) (0,0)--(3,4);
\draw (0,0) node[anchor=east] {\tiny$O$};
\draw (0,2) node[anchor=east] {\tiny$A_1$};
\draw (2,6) node[anchor=south east] {\tiny$A_2$};
\draw (6,8) node[anchor=south] {\tiny$A_3$};
\draw (10,6) node[anchor=west] {\tiny$C$};
\draw (10,4) node[anchor=west] {\tiny$B_2$};
\draw (6,0) node[anchor=north] {\tiny$B_1$};
\draw[->] (0,0) -- (10,0)
node[above] {\tiny $x$};
\draw[->] (0,0) -- (0,7)
node[left] {\tiny $y$};
\draw[red] (0,1)--(1,3)--(3,4)--(5,3);\fill[blue] (0,1) circle[radius=2pt];\fill[red] (5,3) circle[radius=2pt];
\begin{scope}[shift={(0,.5)}]
\draw[red] (0,1)--(1,3)--(3,4)--(5,3);\fill[blue] (0,1) circle[radius=2pt];\fill[red] (5,3) circle[radius=2pt];
\end{scope}
\begin{scope}[shift={(0,1)}]
\draw[red] (0,1)--(1,3)--(3,4)--(5,3);\fill[blue] (0,1) circle[radius=2pt];\fill[red] (5,3) circle[radius=2pt];
\end{scope}
\begin{scope}[shift={(0.25,1.5)}]
\draw[red] (0,1)--(1,3)--(3,4)--(5,3);\fill[blue] (0,1) circle[radius=2pt];\fill[red] (5,3) circle[radius=2pt];
\end{scope}
\begin{scope}[shift={(.5,2)}]
\draw[red] (0,1)--(1,3)--(3,4)--(5,3);\fill[blue] (0,1) circle[radius=2pt];\fill[red] (5,3) circle[radius=2pt];
\end{scope}
\begin{scope}[shift={(.75,2.5)}]
\draw[red] (0,1)--(1,3)--(3,4)--(5,3);\fill[blue] (0,1) circle[radius=2pt];\fill[red] (5,3) circle[radius=2pt];
\end{scope}
\begin{scope}[shift={(1,3)}]
\draw[red] (0,1)--(1,3)--(3,4)--(5,3);\fill[blue] (0,1) circle[radius=2pt];\fill[red] (5,3) circle[radius=2pt];
\end{scope}
\begin{scope}[shift={(1.5,3.25)}]
\draw[red] (0,1)--(1,3)--(3,4)--(5,3);\fill[blue] (0,1) circle[radius=2pt];\fill[red] (5,3) circle[radius=2pt];
\end{scope}
\begin{scope}[shift={(2,3.5)}]
\draw[red] (0,1)--(1,3)--(3,4)--(5,3);\fill[blue] (0,1) circle[radius=2pt];\fill[red] (5,3) circle[radius=2pt];
\end{scope}
\begin{scope}[shift={(2.5,3.75)}]
\draw[red] (0,1)--(1,3)--(3,4)--(5,3);\fill[blue] (0,1) circle[radius=2pt];\fill[red] (5,3) circle[radius=2pt];
\end{scope}
\begin{scope}[shift={(3,4)}]
\draw[red] (0,1)--(1,3)--(3,4)--(5,3);\fill[blue] (0,1) circle[radius=2pt];\fill[red] (5,3) circle[radius=2pt];
\end{scope}
\end{tikzpicture}
\end{center}
\caption{The red broken lines are $p+L_2$ whose left endpoint $p$ is blue and right endpoint is red.}
\label{fig:sweep}
\end{figure}
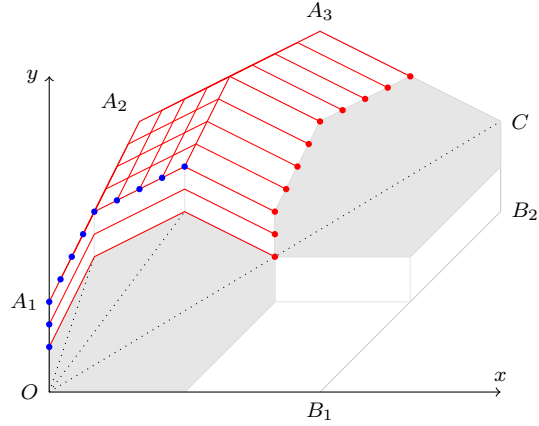

If $i<j$, then $r=p+q=A_{i-1,j-1}+s{\bf v}_i+t{\bf v}_j$ lies in $\mathcal{P}_{i,j}$. 

If $i=j$ and $s+t\le 1$, then  $r=p+q=A_{i-1,i-1}+(s+t){\bf v}_i=A_{i-2,i-1}+1{\bf v}_{i-1}+(s+t){\bf v}_i$ lies in $\mathcal{P}_{i-1,i}$. 

If $i=j$ and $s+t> 1$, then  $r=p+q=A_{i-1,i-1}+(s+t){\bf v}_i=A_{i-1,i}+(s+t-1){\bf v}_i+0{\bf v}_{i+1}$ lies in $\mathcal{P}_{i,i+1}$. 

So we have proved that in all cases, the point $r$ lies in some parallelogram $\mathcal{P}_{ij}$. This proves (a).

\noindent (b) Assume that there exists a point $r\in\mathcal{P}_{ij}\cap\mathcal{P}_{i'j'}$ where $i<j$ and $i'<j'$ and $i\le i'$. Then 
$$r=O+\sum_{k=1}^{i-1}{\bf v}_k+\sum_{k=1}^{j-1}{\bf v}_k+s{\bf v}_i+t{\bf v}_j=O+\sum_{k=1}^{i'-1}{\bf v}_k+\sum_{k=1}^{j'-1}{\bf v}_k+s'{\bf v}_{i'}+t'{\bf v}_{j'}$$
for some $s,t,s',t'\in [0,1]$.  

If $i'=i$ and $j'>j$, then $s{\bf v}_i+t{\bf v}_j=\sum_{k=j}^{j'-1}{\bf v}_k+s'{\bf v}_i+t'{\bf v}_{j'}$, hence $(s-s'){\bf v}_i=(1-t){\bf v}_j+\sum_{k=j+1}^{j'-1}{\bf v}_k+t'{\bf v}_{j'}$. Note that ${\bf v}_i, {\bf v}_j,{\bf v}_{j+1},\dots,{\bf v}_{j'}$ lie in a half plane and are in the clockwise order. So we must have $j'=j+1$ and $s-s'=1-t=t'=0$. In this case, $r=A_{i-1,j}+s{\bf v}_i\in \overline{A_{i-1,j}A_{ij}}$.

If $j'=j$ and $i'>i$, then by a similar argument we have $i'=i+1$ and $r\in \overline{A_{i,j-1}A_{ij}}$.

If $i'>i$ and $j'>j$, then $s{\bf v}_i+t{\bf v}_j=\sum_{k=i}^{i'-1}{\bf v}_k+\sum_{k=j}^{j'-1}{\bf v}_k+s'{\bf v}_{i+1}+t'{\bf v}_{j+1}$, $(1-s){\bf v}_i+(1-t){\bf v}_j+\sum_{k=i+1}^{i'-1}{\bf v}_k+\sum_{k=j+1}^{j'-1}{\bf v}_k+s'{\bf v}_{i+1}+t'{\bf v}_{j+1}=0$. Note that all ${\bf v}_k$ lie in a half plane, we must have $i'=i+1, j'=j+1$, $1-s=1-t=s'=t'=0$. In this case, $r=A_{ij}$.

If $i'>i$ and $j'<j$, then 
$\sum_{k=j'}^{j-1}{\bf v}_k+s{\bf v}_i+t{\bf v}_j=\sum_{k=i}^{i'-1}{\bf v}_k+s'{\bf v}_{i'}+t'{\bf v}_{j'}$, 
$(1-t'){\bf v}_{j'}+\sum_{k=j'+1}^{j-1}{\bf v}_k+t{\bf v}_j=(1-s){\bf v}_i+\sum_{k=i+1}^{i'-1}{\bf v}_k+s'{\bf v}_{i'}$. Note that ${\bf v}_{i},\dots, {\bf v}_{i'},{\bf v}_{j'},\dots,{\bf v}_{j}$ lie in a half plane and are in the clockwise order, we must have $i'=i+1$, $j'=j-1$, $1-s'=s=1-t=t'=0$. In this case, $r=A_{i,j-1}$.

This proves all cases for (b). 
\end{proof}

\subsection{Construction of broken lines}

Next we construct broken lines $T_D$'s and  $T'_D$'s.

\noindent (a) For every point $D$ on the line segment $\overline{C'C}$, we draw a broken line
$T_D$ that goes to the left until it reaches the boundary of $N(\ff)$, as follows: 

Step 1. First goes in the direction $\overrightarrow{A_\alpha A_{\alpha-1}}$ until it reaches a point $D_{\alpha-1}$ on the west boundary of $\mathcal{P}_{r,\alpha}$ for some $r$;

Step 2. Then goes in the direction $\overrightarrow{A_{\alpha-1} A_{\alpha-2}}$ until it reaches  a point $D_{\alpha-2}$ on the west boundary of $\mathcal{P}_{r,\alpha-1}$;

Step 3. then goes in the direction $\overrightarrow{A_{\alpha-2} A_{\alpha-3}}$ until it reaches  a point $D_{\alpha-3}$ on the west boundary of $\mathcal{P}_{r,\alpha-2}$;

$\cdots$

Step $(\alpha-r)$. Finally, goes in the direction $\overrightarrow{A_{r+1} A_{r}}$ until it reaches   a point $D_{r}$ on the west boundary of $\mathcal{P}_{r,r+1}$. 

If $\overline{DD_{\alpha-1}}$ does not contain the north edge of $P_{r,\alpha}$ for any $r\ge1$, then define $T_D=DD_{\alpha-1}D_{\alpha-2}\cdots D_r$. Note that it now reaches the boundary of $N(\ff)$. 

If $\overline{DD_{\alpha-1}}$ contains the north edge of $P_{r,\alpha}$ for some $r\ge1$, we define $T_D$ to be the broken line consisting of $\overline{DD_{\alpha-1}}$ and the north edges of $P_{r,\alpha-1}$, $P_{r,\alpha-2}$,..., $P_{r,r+1}$, and ending at $A_{r}$.

\noindent (b) Still for every point $D$ on the line segment $\overline{C'C}$, similarly  we define a broken line $T'_D$ that goes down whose linear pieces are parallel to $\overrightarrow{B_i B_{i-1}}$ for $i\le\beta$.
Note that for $D=C'$, $T_{C'}=\frac{1}{2} A_{\alpha}A_{\alpha-1}\cdots A_1=A_{0\alpha}A_{0,\alpha-1}\cdots A_{01}$, and $T'_{C'}=\frac{1}{2} B_{\beta}B_{\beta-1}\cdots B_1=B_{0\beta}B_{0,\beta-1}\cdots B_{01}$.  

\noindent (c) For every $D\in\overline{OC'}\setminus\{O\}$, define $T_D=D_\alpha D_{\alpha-1}\cdots D_1=\frac{||OD||}{||OC||}A_\alpha A_{\alpha-1}\cdots A_1$ (rescaling the boundary broken line from $A_\alpha$ to $A_1$ so that it passes the point $D=D_\alpha$).

\noindent (d) For every $D\in\overline{OC'}\setminus\{O\}$, define $T'_D=D_\beta D_{\beta-1} \cdots D_1=\frac{||OD||}{||OC||}B_\beta B_{\beta-1}\cdots B_1$ (rescaling the boundary broken line from $B_\beta$ to $B_1$ so that it passes the point $D=D_\beta$).
\smallskip

Note that every point on $N(\ff)\setminus \overline{OC}$ lies in a unique $T_D$ or $T'_D$. 

\subsection{Proof of Theorem~\ref{main_thm}}
Let $P$ be a polynomial given by Lemma~\ref{lem:c=p^2 most generalized}, and let $R=F-P^2$.  We have the following generalization of Proposition \ref{sum_PP}.

\begin{proposition}\label{sum_PP generalized} Suppose that $(F,G)$ has generic boundaries with $a = 2$. In the above setting, $R$ must be a constant; that is,  $\ff=P^2+u_0$ for some constant $u_0\in K$. 
\end{proposition}
\begin{proof}
We consider the finite sequence of broken lines $T_1,T_2,\dots,T_s$ such that \\
\noindent $\bullet$ each $T_i$ is either $T_D$ or $T'_D$ for some $D\in\overline{OC}\setminus\{O\}$;\\  
\noindent $\bullet$ $\text{supp}(F)\setminus\{O\}\subset \cup_{i=1}^s T_i$;\\
\noindent $\bullet$ each $T_i$ contains a lattice point;\\
\noindent $\bullet$  and the  distance from $O$ to $T_i\cap \overline{OC}$ is no less than the distance from $O$ to $T_{i+1}\cap \overline{OC}$.

To make the order unique, we assume that if $T_i\cap \overline{OC}=T_{i+1}\cap \overline{OC}=D$, then $T_i=T_D$ and $T_{i+1}=T'_D$. In particular, $T_1=\overline{A_\alpha A_{\alpha-1}}$ (recall that $\alpha>1$) and $T_2=\overline{B_\beta B_{\beta-1}}$ if $\beta>1$. 

We will inductively prove that $\text{supp}(R)\cap  T_i=\emptyset$. 

(a) Let $w=(u,v)$ be the normal direction of the edge $\overline{A_\alpha A_{\alpha-1}}$. Let $d=w\text{-deg}(F_+)$ and $m=w\text{-deg}(P_+)$, so $d=2m$. The assumption $[F,G]\in K$ implies that $\ff_d=P_m^2$. This is well known (for instance, see \cite{A,ApOn,NaBa}), and also follows from Theorem~\ref{Magnus_thm} for $\mu=0$. Thus $\text{supp}(R)\cap \overline{A_\alpha A_{\alpha-1}}=\text{supp}(R)\cap  T_1=\emptyset$. Similarly, $\text{supp}(R)\cap \overline{B_\beta B_{\beta-1}}=\text{supp}(R)\cap  T_2=\emptyset$ if $\beta>1$. 

(b) Assume the inductive hypothesis that  $\text{supp}(R)\cap  (T_1 \cup\cdots\cup T_{i-1})=\emptyset$. We consider the broken line $T_i$. There are two cases.

(b1) Suppose $T_i=T_D$ or $T'_D$ for $D=(d_1,d_2)$ in the segment $C'C$. Without loss of generality assume $T_i=T_D=DD_{\alpha-1}D_{\alpha-2}\cdots D_r$. 

\begin{figure}[h]
\begin{center}
\begin{tikzpicture}[scale=0.60]
\begin{scope}[shift={(0,0)}]
\usetikzlibrary{patterns}
\draw[black!20]  (0,0)--(0,2)--(2,6)--(6,8)--(10,6)--(10,4)--(6,0);
\draw[black!20]  (0,0)--(0,1)--(1,3)--(3,4)--(5,3)--(5,2)--(3,0);
\draw[black!20]  (5,3)--(5,4)--(6,6)--(8,7)--(10,6)--(10,5)--(8,3)--(5,3);
\fill [black!10]  (0,0)--(0,1)--(1,3)--(3,4)--(5,3)--(5,2)--(3,0);
\fill [black!10]  (5,3)--(5,4)--(6,6)--(8,7)--(10,6)--(10,5)--(8,3)--(5,3);
\draw[black!10]  (1,3)--(1,4)--(3,5)--(5,4) (3,4)--(3,5)--(4,7)--(6,6);
\draw[black!10]  (8,3)--(8,2)--(5,2);
\draw[dotted] (0,0)--(10,6) (0,0)--(1,3) (0,0)--(3,4);
\draw[red] (0,1.6)--(1,3.6)--(3,4.6)--(5,3.6)--(5.5,3.33)--(8.1,2.05);
\fill[red] (6.2,3) circle[radius=2pt]; \draw (6.2,3) node[anchor=south west] {\tiny$E_1$};
\fill[red] (8.1,2.05) circle[radius=2pt]; \draw (8.2,2) node[anchor=south west] {\tiny$E_2$};
\fill[red] (5.5, 3.33) circle[radius=2pt]; \draw (5.5, 3.33) node[anchor=south west] {\tiny$D$};
\fill[red] (0,1.6)  circle[radius=2pt]; \draw (0,1.6) node[anchor=south east] {\tiny$D_r$};
\fill[red] (1,3.6)  circle[radius=2pt]; \draw (1,3.6) node[anchor=south] {\tiny$D_{r+1}$};
\fill[red] (3,4.6)  circle[radius=2pt]; \draw (3,4.6) node[anchor=south] {\tiny$D_{\alpha-1}$};
\fill[red] (5,3.6) circle[radius=2pt]; \draw (5,3.6) node[anchor=south] {\tiny$D_{\alpha}$};
\draw (0,0) node[anchor=east] {\tiny$O$};

\draw[->] (0,0) -- (10,0)
node[above] {\tiny $x$};
\draw[->] (0,0) -- (0,7)
node[left] {\tiny $y$};
\end{scope}
\begin{scope}[shift={(12,0)}]
\usetikzlibrary{patterns}
\draw[black!20]  (0,0)--(0,2)--(2,6)--(6,8)--(10,6)--(10,4)--(6,0);
\draw[black!20]  (0,0)--(0,1)--(1,3)--(3,4)--(5,3)--(5,2)--(3,0);
\draw[black!20]  (5,3)--(5,4)--(6,6)--(8,7)--(10,6)--(10,5)--(8,3)--(5,3);
\fill [black!10]  (0,0)--(0,1)--(1,3)--(3,4)--(5,3)--(5,2)--(3,0);
\fill [black!10]  (5,3)--(5,4)--(6,6)--(8,7)--(10,6)--(10,5)--(8,3)--(5,3);
\draw[black!10]  (1,3)--(1,4)--(3,5)--(5,4) (3,4)--(3,5)--(4,7)--(6,6);
\draw[black!10]  (8,3)--(8,2)--(5,2);
\draw[dotted] (0,0)--(10,6) (0,0)--(1,3) (0,0)--(3,4);
\draw[red] (0,0.6)--(.6,1.8)--(1.8,2.4)--(3,1.8)--(6.2,0.2);
\fill[red] (6.2,0.2) circle[radius=2pt]; \draw (6.2,0.2) node[anchor=south west] {\tiny$E$};
\fill[red] (3,1.8) circle[radius=2pt]; \draw (3,1.8) node[anchor=south west] {\tiny$D=D_\alpha$};
\fill[red] (1.8,2.4) circle[radius=2pt]; \draw (1.8,2.4) node[anchor=south west] {\tiny$D_{\alpha-1}$};
\fill[red] (.6,1.8) circle[radius=2pt]; \draw (.6,1.8) node[anchor=south east] {\tiny$D_{2}$};
\fill[red] (0,0.6) circle[radius=2pt]; \draw (0,0.6) node[anchor=south east] {\tiny$D_{1}$};

\draw (0,0) node[anchor=east] {\tiny$O$};
\draw[->] (0,0) -- (10,0)
node[above] {\tiny $x$};
\draw[->] (0,0) -- (0,7)
node[left] {\tiny $y$};
\end{scope}
\end{tikzpicture}
\end{center}
\caption{Case (b1) on the left; Case (b2) on the right.}
\label{fig:b1 b2}
\end{figure}
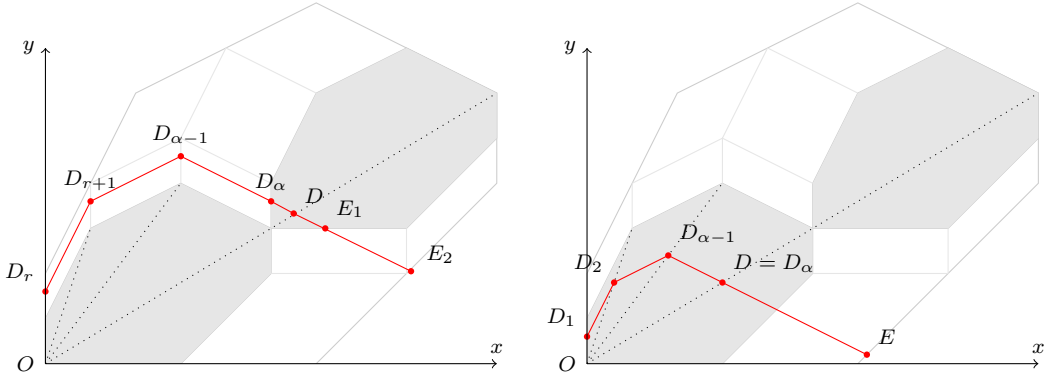

By Proposition~\ref{application_of_Magnus}, we see that 
$P_m^{-1}R_{ud_1+vd_2}\in K[x^{\pm1},y^{\pm1}]$. Denote by $D_\alpha$ the intersection of the line segment $\overline{DD_{\alpha-1}}$ with the boundary of $\mathcal{N}''$.  Let $E_1$  be the intersection point of the (unbounded) half line $\overrightarrow{D_{\alpha}D}$ and the  boundary of $\mathcal{N''}$. Let $E_2$ be the intersection of $\overrightarrow{D_{\alpha}D}$ with the  boundary of $N(F)$. See the picture on the left in Figure \ref{fig:b1 b2}. 

Since $\text{supp}(R)\cap\mathcal{N}''=\emptyset$, we have $\text{supp}(R)\cap \overline{D_{\alpha}E_1}=\emptyset$. Since the lattice points on $\overline{E_1E_2}$ must be in some $T_{\ell}$ for some $\ell<i$, we also have $\text{supp}(R)\cap\overline{E_1E_2}=\emptyset$. Hence $\text{supp}(R_{ud_1+vd_2}) \in \overline{D_\alpha D_{\alpha-1}}\setminus \{D_\alpha\}$, so $\text{len}(R_{ud_1+vd_2})<\text{len}(P_m)$. Then  $\text{supp}(R) \cap \overline{D_\alpha D_{\alpha-1}}=\emptyset$ due to Corollary~\ref{application2_of_Magnus}. 
Applying a similar argument, we see that $\text{supp}(R) \cap \overline{D_{\alpha-1}D_{\alpha-2}}=\cdots=\text{supp}(R) \cap \overline{D_{r+1}D_{r}}=\emptyset$.     Therefore $\text{supp}(R) \cap T_i=\emptyset$.

(b2)  Suppose $T_i=T_D$ or $T'_D$ for $D=(d_1,d_2)\in\overline{OC'}\setminus\{O,C'\}$ . Without loss of generality assume $T_i=T_D=D_\alpha D_{\alpha-1}\cdots D_1$, where $D_\alpha=D$. Let $E$ be the intersection of $\overrightarrow{D_{\alpha-1}D}$ with the  boundary of $N(F)$. See the picture on the right in Figure \ref{fig:b1 b2}. Since the lattice points on $\overline{DE}$ must be in some $T_{\ell}$ for some $\ell<i$, we also have $\text{supp}(R)\cap\overline{DE}=\emptyset$. Hence $\text{supp}(R_{ud_1+vd_2}) \in \overline{D_\alpha D_{\alpha-1}}$, so $\text{len}(R_{ud_1+vd_2})<\text{len}(P_m)$. From here, we get the same conclusion as in (b1).
\end{proof}

\begin{proof}[Proof of Theorem~\ref{main_thm}]
The same as the proof of Corollary~\ref{sum_PP_imply}.
\end{proof}

\end{document}